\documentclass{amsart}

\usepackage[latin1]{inputenc}
\usepackage{amsfonts}
\usepackage{amsmath}
\usepackage{amsthm}
\usepackage{amssymb}
\usepackage{latexsym}
\usepackage{enumerate}

\newtheorem{theorem}{Theorem}[section]
\newtheorem{lemma}[theorem]{Lemma}
\newtheorem{proposition}[theorem]{Proposition}
\newtheorem{corollary}[theorem]{Corollary}

\newtheorem{question}[theorem]{Question}

\newtheorem{definition}[theorem]{Definition}

\numberwithin{equation}{section}

\begin{document}

\newcommand{\cc}{\mathfrak{c}}
\newcommand{\N}{\mathbb{N}}
\newcommand{\Q}{\mathbb{Q}}
\newcommand{\R}{\mathbb{R}}

\newcommand{\PP}{\mathbb{P}}
\newcommand{\forces}{\Vdash}
\newcommand{\dom}{\text{dom}}
\newcommand{\osc}{\text{osc}}

\title{Some topological invariants  and biorthogonal systems
in Banach spaces}

\author{Piotr Koszmider}
\thanks{ The author would like to thank Jes\'us Castillo and Manuel Gonz\'alez for
organizing a stimulating conference at Castro Urdiales in February 2011.} 
\address{Institute of Mathematics, Polish Academy of Sciences,
ul. \'Sniadeckich 8,  00-956 Warszawa, Poland}
\email{\texttt{p.koszmider@impan.pl}}




\date{}

\begin{abstract} 
We consider topological invariants on compact spaces 
related to the sizes of discrete subspaces (spread), densities of subspaces,
Lindel\"of degree of subspaces, irredundant families of clopen sets and others
and look at the following associations between compact topological spaces and Banach spaces:
a compact $K$ induces a Banach space $C(K)$ of real valued
continuous functions on $K$ with the supremum norm; a  Banach
space $X$ induces a compact space $B_{X^*}$, the dual ball with
the weak$^*$ topology. We inquire on how topological invariants on
$K$ and $B_{X^*}$ are linked to the sizes of biorthogonal systems and
their versions in $C(K)$ and $X$ respectively. We gather folkloric facts and
survey recent results like that of Lopez-Abad and Todorcevic that it is
consistent that there is a Banach space $X$ without uncountable
biorthogonal systems such that the spread of
$B_{X^*}$ is uncountable or that of Brech and Koszmider that 
it is consistent that there is a compact space where spread of $K^2$
ic countable but $C(K)$ has uncountable biorthogonal systems.
\end{abstract}
\maketitle

\markright{Topological invariants  and biorthogonal
systems}

\section{Introduction}

It is well known that there are intimate links between Banach spaces and topological
compact spaces. On  one hand, compact spaces provide an important and big class of
Banach spaces of the form $C(K)$, that is, of all real valued continuous functions on
a compact Hausdorff $K$
with the supremum norm
(where, for example, one can isometrically embed all Banach spaces).
On the other hand any Banach space $X$ produces a compact space, namely
the dual unit ball $B_{X^*}$ with the weak$^*$ topology, that is, the topology 
 whose subbasic sets are of the form
$[x, I]=\{x^*\in B_{X^*}: x^*(x)\in I\}$, where $x\in X$ and $I$ is an open interval in $\R$.
One could consider these constructions as parallel to the Stone duality type
constructions (see \cite{koppelberg}) when we replace a Boolean algebra by a Banach space. 
As we do not have the full duality in this case, $B_{C(K)^*}$ is not $K$, nor $X$ is $C(B_{X^*})$,
and we have only the canonical embeddings,  however, many
interactions do have their  Stone duality type analogues.

When a compact space $K$ is  not metrizable, or when a Banach space $X$ is not separable,
many phenomena appear which establish nontrivial correspondences between classes of compact spaces and Banach spaces.  To be more precise, let $\mathcal C$ be a class of compact spaces and $\mathcal X$ be
a class of Banach spaces. We say that $\mathcal C$ and $\mathcal X$ are associated if and only
if $C(K)\in \mathcal X$ whenever $K\in \mathcal C$ and $B_{X*}\in \mathcal C$ whenever
$X\in\mathcal X$. Thus, for example,
WCG Banach spaces are associated with Eberlein compacts,
Asplund generated spaces are associated with Radon-Nikodym compacts, etc. 
In other words, investigating the geometry of $C(K)$ looking at the topological
properties of $K$, or investigating the geometry of a general Banach space $X$
with the help of the topological properties of $B_{X^*}$ is quite fruitful in the nonseparable context.

\begin{definition}
If $X$ is a Banach space and $X^*$ is its dual,
then $(x_i,x_i^*)_{i\in I}\subseteq X\times X^*$ is called a biorthogonal
system if and only if for each $i,j\in I$ we have $x_i^*(x_i)=1$ and $x^*_i(x_j)=0$ if $i\not=j$.

\end{definition}

The importance of biorthogonal systems in the theory of Banach spaces
is related to the fact that any kind of a basic sequence 
in a Banach space $X$ must be the $X$ part of a biorthogonal system, because taking
its coordinate should be a linear functional satisfying the above definition (see
e.g., \cite{singer}, \cite{biorthogonal}). 

In this note we are motivated by a growing amount of results related to cardinalities
of biorthogonal systems and their versions in nonseparable Banach spaces. Our
purpose is to compare them in a partial survey with a well-established body of results concerning certain
topological invariants on $K$s in the case
of Banach spaces of the form $C(K)$ and on $B_{X^*}$ in the case of
a general Banach space $X$. By a  topological  invariant we will mean
a way of associating a cardinal number to a topological space, in most cases a compact space.
In a similar way we define a Banach space invariant. 
The simplest examples of invariants are the weight $w(K)$ of a compact space $K$ and
the density character $dens(X)$ of the Banach space $X$. Since we have the equalities $w(K)=dens(C(K))$
and $dens(X)=w(B_{X^*})$, they are as nicely associated as they can.  We will consider
two groups of invariants, the first, emerging in Banach space theory is related to the
cardinalities of biorthogonal systems. The other group of topological invariants are sometimes
called versions of independence (\cite{vandouwen}) and include spread, hereditary density
and hereditary Lindel\"of degree but also the irredundance is very relevant and we consider it.

To be more specific we need to recall a long list of definitions of these 
invariants and introduce some invariants related to the cardinalities of 
biorthogonal systems and their versions. First, one  can define several weaker versions of biorthogonal systems (\cite{kunenshelah}):

\begin{definition} Suppose $0< \varepsilon<1$.
If $X$ is a Banach space and $X^*$ is its dual,
then $(x_i,x_i^*)_{i\in I}\subseteq X\times X^*$ is called an $\varepsilon$-biorthogonal
system if and only if for each $i,j\in I$ we have $x_i^*(x_i)=1$ and $|x^*_i(x_j)|<\varepsilon$ if $i\not=j$.
Any $\varepsilon$-biorthogonal
system is also called an almost biorthogonal system.
\end{definition}

Now we are in the position to define several invariants on Banach spaces:

\begin{definition}
Let $X$ be a Banach space. Define the following:
\begin{itemize}
\item $biort(X)=\sup\{|I|:\  \hbox{there is a biorthogonal system}\
(x_i,x_i^*)_{i\in I}\subseteq X\times X^*\},$
\item $biort^\varepsilon(X)=\sup\{|I|:\  \hbox{there is an $\varepsilon$-biorthogonal system}\
(x_i,x_i^*)_{i\in I}\subseteq X\times X^*\},$
\item $abiort(X)=\sup\{|I|:\  \hbox{there is an almost biorthogonal system}\
(x_i,x_i^*)_{i\in I}\subseteq X\times X^*\}.$
\end{itemize}
\end{definition}

$\varepsilon$-biorthogonal and almost biorthogonal systems were introduced in
\cite{kunenshelah} where their relationships with more geometric phenomena 
in  Banach spaces were investigated.

If $K$ is compact, then $M(K)$ denotes the Banach space of Radon measures
on $K$, i.e., countably additive, regular, signed Borel measures on $K$ with the variation norm.
Here we should keep in mind the Riesz representation theorem
which says that  Radon measures on $K$ isometrically represent all linear
functionals on Banach spaces of the form $C(K)$, that is there is an isometry
between $M(K)$ and $C(K)^*$ which associates the integration with respect to a given measure
to a given measure. Pointwise measures concentrated on $x\in K$ will be denoted by
$\delta_x$.
In the case of Banach spaces of the form $C(K)$ we introduce the following:

\begin{definition} Let $K$ be a compact space, $n\in\N$  and $\mu\in M(K)$.
We say that $\mu$ is $n$-supported if and only if there are $x_1, ...,x_n\in K$ and
$a_1, ..., a_n\in\R$ such that
$$\mu=a_1\delta_{x_1}+ ...a_n\delta_{x_n}.$$
We say that a biorthogonal system
$(f_i,\mu_i)_{i\in I}\subseteq C(K)\times M(K)$ is $n$-supported if and
only if all measures $\mu_i$ are $n$-supported.
We define
$$biort_n(C(K))=\sup\{|I|: \hbox{there is a biorthogonal system}$$
$$(f_i,\mu_i)_{i\in I}\subseteq C(K)\times M(K)\   \hbox{which is n-supported}\}.$$
\end{definition}
One kind of $2$-supported biorthogonal systems appear more often than  others,
we will follow \cite{mirnaistvan} with the terminology concerning these systems:
\begin{definition} Let $K$ be a compact space, $\kappa$ a cardinal
and $x_\alpha, y_\alpha\in K$ for all $\alpha<\kappa$. A biorthogonal system in
the Banach space $C(K)$ of the form $(f_\alpha, \delta_{x_\alpha}-\delta_{y_\alpha})_{\alpha<\kappa}$
is called a nice biorthogonal system. The supremum of cardinalities
of nice biorthogonal systems will be denoted by $nbiort_2(K)$.
\end{definition}

Now we recall definitions of several topological invariants. 
Suppose $K$ is an infinite compact space, then
we consider the following topological invariants of $K$:
\begin{itemize}
\item $w(K)=\inf\{|\mathcal B|: \mathcal B$ is an open basis for $K\}$,
\item $d(K)=\inf\{ |D|: \ D\subseteq K\ \hbox{\rm where D is dense in}\ K\}$,
\item $L(K)=\inf\{\kappa:$ every open cover of $K$ has a subcover of cardinality $\kappa\}$,
\item $ind(K)=sup\{\kappa:\ \hbox{\rm
there is a continuous surjection}\ \phi:K\rightarrow [0,1]^\kappa\}$,
\item $s(K)=\sup\{|D|:   D\subseteq K\ \hbox{\rm where D is discrete in}\ K\}$,
\item $hd(K)=\sup\{d(X): X\subseteq K\}$,
\item $hL(K)=\sup\{L(X): X\subseteq K\}.$
\end{itemize}
We call them weight, density, Lindel\"of degree, independence, spread, hereditary density,
hereditary Lindel\"of degree respectively. Most of these functions can be redefined
in a uniform language of versions of independence (see \cite{vandouwen}).
For Banach space theoretic aspects of the independence itself see \cite{plebanekind}.
We left the tightness as we will not consider it in this note, however it is an invariant which is extensively used
in the context of the dual ball of a Banach spaces (see e.g. \cite{tightness})
also in relation to biorthogonal systems (\cite{stevobio}).
We recall some well-known inequalities
concerning the above invariants, most of them can be found in \cite{hodel}:
\begin{theorem} Suppose $K$ is a compact space, then
the following hold:
\begin{itemize}
\item $ind(K)\leq s(K)\leq hd(K), hL(K)\leq w(K)$,
\item $w(K)\leq 2^{s(K)}$,
\item $hd(K)\leq s(K)^+$,
\item $hd(K)\leq s(K^2)$ (\cite{juhaszspread}).
\end{itemize}
\end{theorem}
Of course we leave many deep inequalities and equalities with other functions
like $h\pi w$ or $h\pi\chi$ as we will not make use of them  in this note, for some of
them see \cite{hodel}. However one important thing will be:

\begin{proposition}\label{proleftright} Suppose that $K$ is a regular topological space. 
$hd(K)$ is the supremum of cardinalities of left-separated sequences in $K$,
that is sequences $\{x_\alpha:\alpha<\kappa\}$ for which there are open
$U_\alpha\subseteq K$ satisfying $x_\alpha\in U_\alpha$ and $x_\alpha\not\in U_\beta$
for $\alpha<\beta<\kappa$.
$hL(K)$ is the supremum of cardinalities of right-separated sequences in $K$,
that is sequences $\{x_\alpha:\alpha<\kappa\}$ for which there are open
$U_\alpha\subseteq K$ satisfying $x_\alpha\in U_\alpha$ and $x_\beta\not\in U_\alpha$
for $\alpha<\beta<\kappa$.
\end{proposition}

Unexplained terminology and notation should be fairly standard. If $A$ is
a Boolean algebra, then $K_{A}$ denotes its Stone space, that is 
the space of ultrafilters on $A$ with the topology whose subbasic sets
are of the form $[a]=\{u\in K_{A}: a\in u\}$. All totally disconnected compact space are
the Stone spaces of some Boolean algebras and this class plays an important role in the
theory of Banach spaces of the form $C(K)$.
All Banach spaces considered here are infinite dimensional, over the reals and all compact
spaces are assumed to be infinite. $K$ usually stands for a compact space,
$X$ for a Banach space and $\kappa$ for an infinite cardinal. CH stands for
the continuum hypothesis.
Given a Banach space $X$, when talking about its dual space $X^*$ as a topological
space, we will always consider the weak$^*$ topology mentioned above.

\begin{lemma}\label{ballstar} Suppose that $X$ is an infinite dimensional Banach 
space and $\psi$ is one of the cardinal invariants
$s$, $hL$, $hd$. Then $\psi(X^*)=\psi(B_{X^*})$.
\end{lemma}
\begin{proof} 
All of the functions are suprema of sizes of some sets in the space
by \ref{proleftright}. So $\psi(X^*)\geq\psi(B_{X^*})$.
Subsets of discrete sets are again discrete. Subsequences of
left-  \ or right-separated sequences are again left or right-separated respectively.
$X^*=\bigcup_{n\in N} nB_{X^*}$ and multiplying by $1/n$ is a homeomorphism
of $nB_{X^*}$ onto $B_{X^*}$ for each $n$. 
Any subset $A$ of $X^*$ can be divided into countably many parts $A_n=X\cap nB_{X^*}$
and each of these parts has a homeomorphic copy ${1\over n}A_n$ in $B_{X^*}$, and
the supremum over sizes of ${1\over n}A_n$ is equal to the size of $A$.
\end{proof}

The recent results concerning the subject matter of this note are included in
\cite{brechkoszmider}, \cite{brechkoszmider2}, \cite{mirnaistvan},
\cite{piotrrolewicz}, \cite{stevogeneric}, \cite{stevobio}.
The last two papers consider the Banach spaces of the form $C(K)$
as a lateral theme, go far beyond biorthogonal systems and investigate 
various kinds of uncountable basic sequences and other related topics. These exciting new advances 
(see also \cite{stevobasic}) are not, however, the
subject of this survey.

Many of the constructions mentioned in this paper are not absolute, that is 
the usual axioms of mathematics (ZFC) are not sufficient to carry them out.
Many follow from additional axioms which were shown to be equiconsistent
with ZFC (they do not lead to contradiction if ZFC does not lead itself) but another
group was established only using the method of forcing.
The readers less familiar with these maters should consult when needed, for example, the textbook
\cite{kunen}.
On the level of $\omega$ and $\omega_1$ this lack of absolutness is known to be unavoidable as
it is shown in most cases that some other axioms or forcing
arguments imply the nonexistence of the constructions.
However it is still unknown if one can construct (in ZFC) Banach spaces (of the form $C(K)$)
where any of the functions $biort(X)$, $s(B_{X^*})$, $s(K^\omega)=hd(K^\omega)$ are different.

All proofs of this note are natural enough that they can be considered as a folclore.
However we think that in the case of interactions of two disciplines, in this case,  Banach space theory and
set-theoretic topology,  simple facts linking apparently unrelated topics may be
known only to isolated groups of reseachers. 
The aim of this note is to propose breaking this isolation and sharing some
list of open problems and simple links among them and the published results.

\section{Biorthogonality and the topological weight}

The most basic topological invariant is the weight of the space. So, we start by
asking what is the relation between the weight of $K$ and the biorthogonality of $C(K)$ or
the weight of $B_{X^*}$ and the biorthogonality of a Banach space $X$. It is well
known that $w(B_{X^*})=dens(X)$ and $w(K)=dens(C(K))$, so trivially we have
$biort(X)\leq w(B_{X^*})$ and $biort(C(K))\leq w(K)$  and we are asking if there could be
big Banach spaces with small biorthogonality.

\begin{theorem} Let $X$ be  an infinite
dimensional Banach space. Then $dens(X)\leq 2^{biort(X)}$.
\end{theorem}
\begin{proof} 
Suppose that $dens(X)>2^\kappa$. We will construct a biorthogonal system of size $\kappa^+$
 by transfinte induction using the Hahn-Banach theorem.

First note that if $\mathcal F\subseteq X^*$  separates the points of $X$
(i.e.,  for distinct $x,y\in X$ we have $f(x)\not=f(y)$ for some $f\in \mathcal F$,
equivalently for every $x\in X\setminus\{0\}$ there is $f\in \mathcal F$ such that $f(x)\not=0$)
then it is of cardinality  bigger than $\kappa$. 
This is due to the fact that if $\mathcal F$ separates the points, then
the function $F$ which sends $x\in X$ into $\R^{\mathcal F}$ defined by
$F(x)(f)=f(x)$ is injective, which would give that $dens(X)\leq |X|\leq |\R^{\mathcal F}|
=(2^\omega)^\kappa=2^\kappa$.
contradicting the hypothesis about the density of $X$.

This fact  implies that given a family $\mathcal F\subseteq X^*$ of cardinality
not bigger than $\kappa$, the norm closed linear subspace of $X$ defined as
$$Y=\bigcap \{ker(f): f\in \mathcal F\}$$
must have density bigger than $\kappa$ because if not, then $\mathcal F$ together
with $\kappa$ functionals separating the points of $Y$ would separate the the points of $X$.

Suppose we have constructed a part of the biorthogonal system
$(x_\alpha, x^*_\alpha)_{\alpha<\lambda}\subseteq X\times X^*$ for some $\lambda<\kappa^+$.
As $\bigcap\{ker(x^*_\alpha): {\alpha<\lambda}\}$ must
have density bigger than $\kappa$, there is $x_\lambda\in X\setminus\{0\}$
such that $x^*_\alpha(x_\lambda)=0$ for all $\alpha<\lambda$ and $x$ is not
in the norm closure of $\{x_\alpha:\alpha<\lambda\}$ which has density at most $\kappa$.
Using the Hahn-Banach theorem we find a continuous functional $x_\lambda^*\in X^*$ such
that $x^*_\lambda(x_\lambda)=1$ and $x^*_\lambda(x_\alpha)=0$ for $\alpha<\lambda$.
This way we continue the construction up to $\kappa^+$.

\end{proof}
The above result was known some decades ago to
W. Johnson (\cite{mujica}). For Banach spaces of the form $C(K)$ we have the
following:

\begin{theorem} Suppose $K$ is a compact space, then 
 $$dens(C(K))\leq 2^{biort_1(C(K))}.$$ 
\end{theorem}
\begin{proof} By \ref{spread1} $biort_1(C(K))=s(K)$.
So, use the well known inequality for
compact spaces $w(K)\leq 2^{s(K)}$ (7.7 \cite{hodel})
\end{proof}

By, now, all known examples of Banach spaces where the density is not equal to
the biorthogonality are not absolute.

\begin{theorem} The following are consistent:
\begin{enumerate}
\item (K. Kunen) There is a compact $K$ such that $dens(C(K))=\omega_1=2^\omega$
but $biort(C(K))=\omega$,
\item (C. Brech, P.Koszmider) There is a compact $K$ such that $dens(C(K))=\omega_2=2^\omega$
but $biort(C(K))=\omega$,
\item (S. Todorcevic) Whenever $|A|>\omega$, then $nbiort_2(C(K_{A}))>\omega$
\item (S. Todorcevic) Whenever $dens(X)>\omega$, then $biort(X)>\omega$.
\end{enumerate}
\end{theorem}

(1) Was first proved by K. Kunen (see \cite{negrepontis}) using the Ostaszewski
type space. Ostaszewski's original construction assumed $\diamondsuit$ (see \cite{kunen})
like a construction by S. Shelah of  a Banach space $X$ with
$\omega=biort(X)<dens(X)=\omega_1$ (\cite{shelah2}).
However Kunen used a weaker assumption of  CH.  This assumption can be weakened 
further to
${\mathfrak b}=\omega_1$ as was shown by Todorcevic (\cite{stevopartitions} 2.4.).

(2) was constructed in \cite{brechkoszmider} using forcing. (3) and (4) are results
obtained in \cite{stevobio} assuming Martin's axiom and the negation of CH and
Martin's Maximum respectively. The natural questions in the context of the
above results are:

\begin{question}
Is it consistent that for an arbitrary Banach space $X$ we have 
$$biort(X)=dens(X)?$$
Or there is an absolute example of a Banach space satysfying $biort(X)<dens(X)?$
\end{question}
\begin{question}
Is it consistent that there is a Banach space $X$
such that $biort(X)=\omega$,  and
$dens(X)>\omega_2$?
\end{question}
One could try to construct a an exemple as above similar to
the Kunen space or the example from \cite{brechkoszmider}, i.e.,  of the form $C(K)$ where
$K$ is scattered and $K^n$ is hereditarily separable for all $n\in \N$ (see section
on hereditary density). Then
the problem becomes more difficult than a well-known open problem
if there is a thin very-tall Boolean algebra of height  $\omega_3$.  However
the example could be very different.

\begin{question} Assume Martin's axiom and the negation of CH. 
Let $A$ be a Boolean algebra and $K$ be a compact space.
Does any of the following statements follow:
\begin{itemize}
\item Whenever $|\mathcal A|<2^\omega$, then $biort(C(K_{\mathcal A}))=|\mathcal A|$?
\item Whenever $w(K)>\omega$, then $biort(C(K))>\omega$?
\item Whenever $w(K)<2^\omega$, then $biort(C(K))=dens(C(K))$?
\end{itemize}
\end{question}

\begin{question} Does Martin's axiom with the negation of CH imply
that every nonseparable Banach space has an uncountable biorthogonal system?
\end{question}

\section{Biorthogonality and spread}

When one looks at the definition of the biorthogonal system in the context of the
weak$^*$ topology, the first thing one notes is the fact that the $X^*$ part of the 
system forms a discrete set of the dual, in terms of our cardinal invariants this is
the following:
\begin{proposition}\label{spread0} Suppose $X$ is a Banach space. Then
$$biort(X)\leq s(B_{X^*}).$$
\end{proposition}
\begin{proof} The weak$^*$ open sets $U_\alpha=\{x^*: x^*(x_\alpha)>1/2\}$
separate $x^*_\alpha$ from the remaining $x_\beta^*$s.
\end{proof}
On the $C(K)$ level we have the following:
\begin{proposition}\label{spread1} Suppose that $K$ is a compact space, then
$$biort_1(C(K))=s(K)$$
\end{proposition}
\begin{proof} Suppose $(x_i: i\in I)$ is a discrete subspace of $K$. This means that
$x_i\not\in\overline{\{x_j: j\not=i\}}$, so we can find a continuous function $f_i:K\rightarrow [0,1]$
such that $f_i(x_i)=1$ and $f_i(x_j)=0$ for all $j\in I\setminus\{i\}$. So $(f_i, \delta_{x_i})_{i\in I}$
forms a $1$-supported biorthogonal system.

Now suppose that $(f_i, a_i\delta_{x_i})_{i\in I}$ is a biorthogonal system. Then none of
$a_i$s can be $0$, hence $f_i(x_j)=0$ and $f_i(x_i)=1/a_i$. consider open sets $U_i=
\{x\in K: f_i(x)> 1/2a_i\}$. We see that $x_i\in U_i$ and $x_j\not \in U_i$ for $j\not=i$, that
is $\{x_i: i\in I\}$ is discrete. 
\end{proof}
An argument similar to that from the proof of \ref{spread0} can be applied for
almost biorthogonal systems, so putting together various functions we obtain the
following:

\begin{corollary}\label{corospread} Let $K$ be a compact space,  $X$ be a Banach space,
 $n\in \N\setminus\{0\}$ and $0<\varepsilon<1$. Then 
\begin{enumerate}
\item $s(K)\leq biort_n(C(K))\leq biort_{n+1}(C(K))\leq
biort( C(K)),$
\item $ biort(X)\leq biort^{\varepsilon}(X)\leq abiort (X)\leq s(B_{X^*}).$
\end{enumerate}
\end{corollary}

So, the spread of $K$ is a lower bound of $biort(C(K))$ and $s(B_{X^*})$ is an upper
bound of $biort(X)$. The split interval $``[0,1]"$ satisfies
$s(``[0,1]")=\omega$ and $biort(``[0,1]")=2^\omega$ (see \cite{finetgodefroy}),
 so the lower bound is not too tight. 
However by now, there is no absolute example where $biort(X)$ and $s(B_{X^*})$
are different.
Although we have the following:

\begin{theorem}[Lopez-Abad, Todorcevic] It is consistent that 
there is a Banach space such that $\omega=biort(X)<s(B_{X^*})=\omega_1$.
\end{theorem}
\begin{proof} Use any of the examples 4.1. or 4.2. of \cite{stevogeneric} 
where a Banach space $X$ is constructed such that
$\omega=biort(X)<abiort(X)=dens(X)=\omega_1$. Now apply
\ref{corospread}. 
\end{proof}

\begin{proposition}\label{spreadpowerabi} Let $K$ be a compact space and 
$n\in\N$, then
$$s(K^n)\leq biort^{1-{1\over n}}(C(K)).$$
In particular, we have that $s(K^n)\leq abiort(C(K))$ for every $n\in \N$.
\end{proposition}
\begin{proof}
Given a discrete set in $K^n$ of cardinality
$\kappa$ for some cardinal $\kappa$, we will construct
an $(1-{1\over n})$-biorthogonal system of the same cardinality.
We may assume that
$n$ is minimal such that there is a discrete set in $K^n$ of cardinality $\kappa$.

Suppose that for $\alpha<\kappa$ the points $(x_1^\alpha, ..., x_n^\alpha)$s form
a discrete set  in $K^n$ as witnessed by open neighbourhoods $U_1^\alpha\times 
...\times U_n^\alpha\subseteq K^n$,
i.e., $(x_1^\alpha, ..., x_n^\alpha)\in U_1^\alpha\times 
...\times U_n^\alpha$ and for every distinct $\alpha,\beta<\kappa$
we have that $x^\beta_i\not\in U^\alpha_i$ for
some  $i\leq n$.  By the minimality of $n$, we may assume that the points $x_1^\alpha, ..., x_n^\alpha$
are distinct for each $\alpha<\kappa$, and so, we may assume that the sets
$U_1^\alpha, ..., U_n^\alpha$ are pairwise disjoint for each $\alpha<\kappa$.
Consider functions $f^\alpha_i: K\rightarrow [0,1]$ such that $f^\alpha_i(x)=0$
if $x\not\in U^\alpha_i$ and $f_i^\alpha(x^\alpha_i)=1$. Let
$$f_\alpha=f^\alpha_1+ ...+ f^\alpha_n, \ \  \mu_\alpha={1\over n}(\delta_{x^\alpha_1}
+ ...+\delta_{x^\alpha_n}).$$
It follows that $\mu_\alpha(f_\alpha)=1$ and $|\mu_\alpha(f_\beta)|\leq 1-{1\over n}$
if $\alpha\not=\beta$, as required.
\end{proof}
\begin{question}
Is there (consistently) a compact space $K$ such that
$$biort(C(K))< s(K^n)$$
for some $n\in\N$?
\end{question}
\begin{question}
Is there (consistently) a compact space $K$ such that
$$s(K^n)<biort(C(K))$$
for all $n\in \N$?
\end{question}
Here we have some positive result:
\begin{theorem}(Brech, Koszmider \cite{brechkoszmider2})
For every $n\in \N$ it is consistent that there is a compact
space $K$ such that
$$\omega=hd(K^n)=s(K^n)<biort(C(K))=\omega_1.$$
\end{theorem}
\begin{question} Suppose $K$ is compact.
Is it true that 
$$biort_n(C(K))=s(K^n)$$
for all (some) $n>1$?
\end{question}
\begin{question} Is it consistent that
$s(B_{X^*})=biort(X)$ for every Banach space $X$?
\end{question}

\section{Biorthogonality and hereditary density}

For many decades the only example of a nonseparable $C(K)$ space 
without uncountable biorthogonal systems was the Kunen space.
Here $K$ is scattered and such that $K^n$ is hereditarily separable 
for every $n\in K$. The fact that $K$ is scattered gives that the dual space
is isometric to $l_1(K)$ and hence is accessible to our sight.  Now 
quite useful is the following:

\begin{lemma} Suppose that $K$ is a scattered compact space. Then
$$hd(K^\omega)=hd(B_{C(K)^*}^\omega).$$
In particular $K^n$ is hereditary separable for each $n\in \N$, if and only if
$(C(K)^*)^n$ and so $B_{C(K)^*}^n$ is hereditary separable for every $n\in\N$.
\end{lemma}
\begin{proof} 
We note that by \ref{proleftright}, we have that $hd(Z^\omega)=\sup\{hd(Z^n):
n\in \N\}$ for a regular space $Z$.
There is a homeomorphic embedding of $K$ into $B_{{C(K)}^*}\subseteq {C(K)}^*$
sending $x\in K$ to $\delta_x$, so the backward implication is clear.
For the forward implication, using
 \ref{proleftright}  given a left-separated sequence 
$\mu_\alpha=(\mu^1_\alpha, ..., \mu^n_\alpha)_{\alpha<\kappa}$
of a regular length $\kappa$
in the $n$-th power of
$B_{X^*}\subseteq X^*=l_1(K)$, we will produce a left-separated sequence
of length $\kappa$ in some finite power of $K$. This is enough since regular cardinals are
unbounded in singular cardinals.

 Let $a_j^{\alpha i}\in \R$ and $x_j^{\alpha i}\in K$ be such that
$$\mu^i_\alpha=\sum_{j\in \N} a_j^{\alpha i}\delta_{x_j^{\alpha i}}.$$
We may assume that the neighbourhoods as in \ref{proleftright} which witness
the fact that $\mu_\alpha$s form a left-separated sequence are of the form
$$U_\alpha=\{(\mu^1, ..., \mu^n): \forall i\leq n\forall l\leq k  \int f^{il}_\alpha d\mu^i\in I_{il}\}$$
for some $f^{il}_\alpha\in C(K)$ and some open intervals $I_{il}\subseteq \R$ for $i\leq n$ and 
$l\leq k$ for some $k\in\N$.
The integer $k$  and the interval $I_{il}$ may be fixed for all $\alpha$s because
the same values will be repeated $\kappa$ many times by the regularity of $\kappa$.
By the same argument we may assume that
the norms of all functions $f^{il}_\alpha$ are bounded by some positive real $M$.

Using the fact that sets of reals have strong condensation points
and regularity of $\kappa$ we may assume that there are
open subintervals  $J_{il}$  of $I_{il}$ and
there is $\varepsilon>0$ satysfying the following for all valued of indices:
$$(\int f^{il}_\alpha d\mu^i_\alpha-2\varepsilon, \int f^{il}_\alpha d\mu^i_\alpha+2\varepsilon)
\subseteq J_{il}\ \ \hbox{and}\ \  (\min(J_{il})-2\varepsilon, \max(J_{il})+2\varepsilon)
\subseteq I_{il}.$$
Let $U_\alpha'$ be sets defined as $U_\alpha$s with $I_{il}$ replaced by $J_{il}$.
In particular $U_\alpha'$s witness the fact that $\mu_\alpha$s form a left-separated sequence.
Again using the regularity of $\kappa$ (in fact only the part that $cf(\kappa)\not=\omega)$
we may assume that there are $j_i\in\N$ such that 
$\sum_{j>j_i} |a_j^{\alpha i}|<\varepsilon/M$ for all $\alpha<\kappa$
and $i\leq n$.
Now define $\nu^i_\alpha=\sum_{j\leq j_i} a_j^{\alpha i}\delta_{x_j^{\alpha i}}.$
It is clear that for each $i\leq n$  and each $l\leq k$ we have
$$|\int f^{il}_\alpha d\mu^i_\alpha-\int f^{il}_\alpha d\nu^i_\alpha|
\leq \sum_{j>j_i} |a_j^{\alpha i} f^{il}_\alpha(x_j^{\alpha i})|\leq\varepsilon.$$
So, if we define $\nu_\alpha=(\nu_\alpha^1, ...,\nu_\alpha^n)$,  in particular we have that
$U_\alpha$s witness the fact that $\nu_\alpha$s form a left-separated sequence.
Finally define
$$x_\alpha=(x_1^{\alpha1}, ..., x_{j_1}^{\alpha1}, ..., x_1^{\alpha n}, ..., x_{j_n}^{\alpha n})$$
and 
$$V_\alpha=V_1^{\alpha1}\times ..., \times V_{j_1}^{\alpha1}\times
 ...\times V_1^{\alpha n}\times ...\times V_{j_n}^{\alpha n},$$
where given $\alpha\leq \kappa$,
 $i\leq n$ and $j\leq j_i$ 
whenever
 $y\in V_j^{\alpha i}$ then $|f_\alpha^{il}(x_j^{\alpha i})-f_\alpha^{il}(y)|<\varepsilon/j_i$.
for each $l\leq k$.
Now if $x=(x_1^{1}, ..., x_{j_1}^{1}, ..., x_1^{n}, ..., x_{j_n}^{n})\in V_\alpha$, then
for each $i\leq n$ and $l\leq k$ we have 
$$|\sum_{j\leq j_i} a_j^{\alpha i} f^{il}_\alpha(x_j^{\alpha i}) -
\sum_{j\leq j_i} a_j^{\alpha i} f^{il}_\alpha(x_j^{ i})|\leq  \varepsilon,$$
because $\sum_{j=0}^{\infty} |a_j^{\alpha i}|\leq 1$ because $\mu_\alpha$ is in the dual unit ball,
so 
$$(a_1^{\alpha 1}\delta_{x_1^{1}}+ ...+ a_{j_1}^{\alpha 1}\delta_{x_{j_1}^{1}}, ...,
a_1^{\alpha n} \delta_{x_1^{n}}+ ...+ a_{j_n}^{\alpha n}\delta_{x_{j_n}^{n}})\in U_\alpha.$$
This implies that $x_\alpha\not\in V_\beta$ for $\alpha<\beta$, because this would give that
$\nu_\alpha\in U_\beta'$. Hence $x_\alpha$s form a left-separated sequence in a finite
power of $K$.

\end{proof}

So $B_{C(K)^*}$ for $K$ being the Kunen space has
 countable spread, and so no uncountable biorthogonal system by \ref{spread0}.
For scattered space $K$ we have $hL(K)=w(K)$ (see \cite{monk} 25.130) and so we have
$\omega=hd(K)<hL(K)=\omega_1$. It is worthy to note the following:

\begin{lemma}\label{sneider} Suppose that $K$ is a compact space, then
$$hL(K^2)=w(K).$$ 
In particular if $K$ is nonmetrizable, then $hL(K^2)$ is uncountable.
\end{lemma}
\begin{proof} It is clear that $hL(K^2)\leq w(K)$. For the opposite inequality consider
 $\Delta=\{(x,x): x\in K\}\subseteq K^2$.
If $hL(K^2)=\kappa$, then $L(K^2\setminus \Delta)\leq\kappa$  so
the open cover of $K^2\setminus \Delta$ by open sets whose closures in $K^2$ are
disjoint from $\Delta$ would have  subcover of size $\kappa$ which would yield that
$K^2\setminus \Delta$ is a union of $\kappa$ many closed sets and so $\Delta$ is 
an intersection of $\kappa$ many open sets $(G_\alpha)_{\alpha<\kappa}$.

Now it is enough to consider a version of a proof of
Sneider's theorem (see \cite{teodor} 5.3.). Consider a family $\mathcal B$
of open sets of $K$ such that for every $\alpha<\kappa$ there are $B_1, ... B_k\in\mathcal B$
such that 
$$\Delta\subseteq \overline{B_1}\times \overline{B_1}\cup ... 
\cup \overline{B_k}\times\overline{B_k}\subseteq G_\alpha.$$
By the compactness, one may assume that $|\mathcal B|\leq\kappa$. Note that $\mathcal B$ is a 
basis for $K$ as it is a pseudobasis. Indeed, otherwise,  if $\bigcap\{B\in \mathcal B: x\in B\}\ni y$
for some $y\not=x$, then $(x,y)\in G_\alpha$ for all $\alpha<\kappa$, contradicting that
fact that $\Delta=\bigcap_{\alpha<\kappa}G_\alpha$. Hence $w(K)\leq\kappa$ as required.

\end{proof}

Recall that a regular space $X$ is called a (strong) $S$-space, if and only if
$X^n$ is hereditarily separable  (for each $n\in \N$) for $n=1$  while $X$ is not hereditarily Lindel\"of.
So Kunen's space is a strong $S$-space.   Martin's axiom with the negation of CH implies that there
are no strong S-spaces or no compact $S$-spaces (see \cite{roitman}).  So, quite natural is the following general question:
\begin{question}
Does the existence of a nonseparable Banach space without uncountable biorthogonal systems
imply the existence of a (strong) compact $S$-space?
\end{question}

In the case of the Kunen space both $K$ and $B_{C(K)^*}$ are strong $S$-spaces, so they are
the natural candidates, however we have two consistent counterexamples. In the proof of the first 
one we will need the following
\begin{lemma}\label{lemmahLhd} Let $X$ be a Banach space considered with the weak topology and
let $X^*$ be its dual considered with weak$^*$ topology. The following hold
for every $n\in \N$ $hL(X^n)\leq\kappa$ if and only if for every $n\in \N$ $hd({X^*}^n)\leq\kappa$.
\end{lemma}
\begin{proof} We will use \ref{proleftright} and will see that using the hypothesis one can obtain left-separated sequences in finite powers of $X^*$ of a regular length $\kappa$ from right-separated sequences in 
finite powers of $X$ of a length $\kappa$ and vice versa. This is enough as
regular cardinals are unbounded in singular cardinals.
Suppose 
$x_\alpha=({x_\alpha^1}^*, ..., {x_\alpha^n}^*)$s for $\alpha<\kappa$ form a left-separated sequence
in $(X^*)^n$, we may assume that it is witnessed as in \ref{proleftright} by open sets $U_\alpha$ of the form
$$U_\alpha=\{({x^1}^*, ..., {x^n}^*): \forall i\leq n\ \forall j\leq m
  \ {x^i}^*(x_\alpha^{ij})\in I^{ij}_\alpha\}$$
for some $x_\alpha^{ij}\in X$, open interval $I^{ij}_\alpha$ and $j\leq m$ for some $m\in \N$
and all $i\leq n$. The integer $m$ is fixed for all $i\leq n$ because we can take 
one which works for all $i\leq n$, it is fixed for all $\alpha$ because $\kappa$ is 
assumed to be uncountable regular, so the same integer is repeated $\kappa$ many times.

Consider the dual open sets in $X^{nm}$ defined as
$$V_\alpha=\{(x^{11}, ... x^{1m}, ..., x^{n1}, ...x^{nm}): \forall i\leq n\ \forall j\leq m
  \ {x^i}^*(x_\alpha^{ij})\in I^{ij}_\alpha\}$$
Define $y_\alpha=(x^{11}_\alpha, ... x^{1m}_\alpha, ...,x^{n1}_\alpha, ... x^{nm}_\alpha)$.
We see that $x_\alpha\in U_\beta$ if and only if $y_\beta\in V_\alpha$ for $\alpha,\beta<\kappa$.
So, $U_\alpha$s witnesses that $\{x_\alpha: \alpha<\kappa\}$
is left-separated if and only if $V_\alpha$s witnesses that $\{y_\alpha: \alpha<\kappa\}$
is right-separated. The other direction is analogous.
\end{proof}

\begin{theorem}[J. Lopez-Abad, S. Todorcevic] It is consistent that there is
a nonseparable Banach space $X$ such that $B_{X^*}$ (nor any
of its power) is not a strong $S$-space
but $abiort(X)=biort(X)=\omega$.
\end{theorem}
\begin{proof} Consider either of the examples 5.1. or 5.2. of \cite{stevogeneric}
which have this property that for some $n\in\N$ the power $X^n$ is not
hereditarily Lindel\"of with respect to the weak topology, so by \ref{lemmahLhd} some
power of $X^*$ is not hereditarily separable with respect to the weak$^*$ topology.
\end{proof}

\begin{theorem} [Koszmider, Lopez-Abad, Todorcevic]
It is consistent there is a nonseparable $C(K)$ without biorthogonal systems
such that $K$ is hereditarily separable and hereditarily Lindel\"of an so is not an S-space.
\end{theorem}
\begin{proof} The space constructed in \cite{piotrrolewicz} is hereditarily
Lindel\"of and $C(K)$ has no biorthogonal systems.
In section 8 of \cite{stevogeneric} it is proved that the space of 
\cite{bellginsburgstevo} has the same properties. The squares of both of theses spaces are
strong $S$-spaces.
\end{proof}

We have some partial positive results however:

\begin{proposition} If $K$ is nonmetrizable compact space such that $C(K)$ has no
uncountable almost biorthogonal systems, then $K^2$ is a strong S-space, in
particular, there exists a compact strong $S$-space.
\end{proposition}
\begin{proof}
Let $K$ be nonmetrizable and such that $C(K)$ has no
uncountable almost biorthogonal systems. It follows from \ref{spreadpowerabi} that
$s(K^n)$ is countable for all $n\in\N$. By  \cite{juhaszspread} $hd(L)\leq s(L^2)$
holds for any compact space, and so we have that $hd(K^n)$ is countable for every $n$.
On the other hand $hL(K^2)$ must be uncountable by \ref{sneider}.
\end{proof}

\begin{proposition}[Todorcevic, Dzamonja-Juhasz]
Suppose that $K$ is a compact space, then
$$hd(K)\geq nbiort_2(C(K)).$$
\end{proposition}

Here in \cite{stevobio}, Todorcevic strengthened a previous result
of Lazar from \cite{lazar} showing the above for $hd(K)=\omega_1$, it was then
generalized by Dzamonja and Juhasz in \cite{mirnaistvan}.

\begin{question} Suppose that $K$ is an nonmetrizable compact space  such that
$C(K)$ has no uncountable biorthogonal systems
is some finite power of $K$ is a (strong) $S$-space? Is $K^2$ an $S$-space?
\end{question}

\section{Biorthogonality and irredundance}

The irredundance of Boolean algebras is a well known and investigated invariant
of Boolean algebras (\cite{monk}). As an introduction to its relation with topological invariants
and biorthogonal systems we propose somewhat general discussion concerning irredundance in various
structures

\begin{definition}Suppose $\mathcal S$ is a class of structures with
fixed families of substructures. Let $S\in \mathcal S$ and $R\subseteq S$.
We say that $R$ is irredundant if and only if for each $r\in R$ 
there is a substructure of $S$ containing $R\setminus\{r\}$ and not containing $r$.
The irredundance $irr_{\mathcal S}(S)$ of $S$ is the supremum of cardinalites
of irredundant subsets of $S$.
\end{definition}

We will consider these notions for Boolean algebras with subalgebras ($\mathcal S=BA$),
Banach spaces with closed linear subspaces ($\mathcal S=BaS$) and Banach algebras of the
form $C(K)$ with closed subalgebras ($\mathcal S=BaA$). Note that in all these cases
the existence of substructures containing $R\setminus\{r\}$ and not containing $r$
means that 
 the substructure generated  by $R\setminus\{r\}$ does not contain $r$.
So if  $A$ is a Boolean algebra then 
$R\subseteq A$ is irredundant if and only if
none of $r\in\mathcal R$ belongs to the Boolean algebra generated by
the remaining elements i.e., $R\setminus\{r\}$, or 
if $K$ is a compact Hausdorff space
${\mathcal F}\subseteq C(K)$ is irredundant if and only if
none of $f\in\mathcal F$  belongs
to the closed subalgebra of $C(K)$ generated by ${\mathcal F}\setminus\{f\}$.

\begin{proposition}\label{variousirr} Suppose $\mathcal A$ is a Boolean algebra and $K$ is a compact
space, then
$$irr_{BA}( A)\leq irr_{BaA}(C(K_{ A}))\leq irr_{BaS}(C(K_{A})),$$
$$irr_{BaA}(C(K))\leq irr_{BaS}(C(K)),$$
\end{proposition}
\begin{proof} If $B$ is a subalgebra of a Boolean algebra $A$, and  $a\in A\setminus B$,
then $\chi_{[a]}$ does not belong to the closed Banach subalgebra generated by
$\chi_{[b]}$s for $b\in B$. Clearly any Banach subalgebra is a Banach subspace.
\end{proof}

Below we will see that two of the above functions are equal to 
invariants previously considered. First note the following proposition which explains why we are talking about the
irredundance:

\begin{proposition} Suppose $X$ is a Banach space and
$(x_\alpha)_{\alpha<\kappa}\subseteq X$. There is 
$(x_\alpha^*)_{\alpha<\kappa}\subseteq X^*$  such that $(x_\alpha, x^*_\alpha)_{\alpha<\kappa}$
is a biorthogonal system  if and only if for every
$\alpha\in\kappa$ the vector $x_\alpha$ does not belong to the norm closed
linear subspace generated  by $\{x_\beta: \beta\not=\alpha\}$.
In particular 
$$irr_{BaS}(X)=biort(X).$$ 
\end{proposition}
\begin{proof} In one direction one considers
$ker(x_\alpha^*)$ as a closed subspace which contains
$x_\beta$s for $\beta\not=\alpha$.  In the other direction one uses the Hahn-Banach theorem
to extend the functional that chooses the $x_\alpha$'s coordinate in 
$$\overline{span(\{x_\alpha:\alpha\not=\beta\})}\oplus\R x_\alpha$$
to the entire space $X$ obtaining $x_\alpha^*$. Here we used the fact the the above
direct sum is a closed subspace of $X$ which follows from the fact that the second factor
is one-dimensional.
\end{proof}

\begin{theorem}\label{irredundantfunctions}
Suppose $K$ is a compact Hausdorff space.
${\mathcal F}\subseteq C(K)$ is irredundant if and only if
for each $f\in\mathcal F$ there are $x_f,y_f\in K$ such that
$f(x_f)-f(y_f)>0$ and $f(x_g)-f(y_g)=0$ for distinct $f,g\in \mathcal F$.
In particular 
$$irr_{BaA}(C(K))=nbiort_2(K)$$
 for every compact space $K$.
\end{theorem}
\begin{proof}
It is clear that  if
for each $f\in\mathcal F$ there are $x_f,y_f\in K$ such that
$f(x_f)-f(y_f)>0$ and $f(x_g)-f(y_g)=0$ for distinct $f,g\in \mathcal F$,
then all functions of the closed 
subalgebra generated by  ${\mathcal F}\setminus\{f\}$ do not
separate $x_f$ from $y_f$, i.e., $f$ does not belong to it.

Now suppose that $f$ is not in  the
closed subalgebra $\mathcal A$ generated by ${\mathcal F}\setminus\{f\}$,
and let us construct $x_f$ and $y_f$ as required.
Consider the equivalence relation $E$  on $K$ defined by
$xEy$ if and only if $g(x)=g(y)$ for all $g\in\mathcal A$.
Let $L$ be the quotient space $K/E$ and $\phi:K\rightarrow L$
be the quotient map. Note that for each $g\in \mathcal A$ there is
a well-defined $[g]: L\rightarrow \R$ such that $[g]\circ \phi=g$,
so by the properties of the quotient topology (see \cite{engelking} 2.4.2.)
$[g]\in C(L)$. Note that $\{[g]: g\in{\mathcal A}\}$ is a closed
subalgebra of $C(L)$ which separates the points of $L$ and contains the constant functions and hence,
by the Weierstrass-Stone theorem it is the entire $C(L)$. 

If $f$ is constant on every equivalence class of $E$, then  there were 
$[f]:L\rightarrow \R$ satysfying $[f]\circ \phi=f$, and then
we would, the same way, have that $[f]\in C(L)$ and so
$[f]=[g]$ for some $g\in \mathcal A$, which would give $f=g$
contradicting the hypothesis about $f$. So $f$ is nonconstant on
some equivalence class of $E$ and so there are $x_f, y_f\in K$
as required

To obtain 
a nice biorthogonal system from an irredundant set of functions via
the first part the theorem just multiply $f\in\mathcal F$ by
$1/(f(x_f)-f(y_f))$.
\end{proof}

Hence, we do not need special notation either for $irr_{BaS}$ nor for
$irr_{BaA}$ and so we will use only $irr_{BA}$ which will be simply denoted
as $irr$ from this point on.
The link between irredundant sets in Boolean algebras and 
biorthogonal systems and spread of the square was first indicated in the literature in the following:

\begin{theorem}\label{heindorff}[Heindorff]
Suppose $\mathcal A$ is a Boolean algebra and $K$ is a compact space, then
$$irr(\mathcal A)\leq nbiort_2(K_{\mathcal A}), \ \ nbiort_2(K)\leq s(K^2).$$
\end{theorem}
\begin{proof}
The first inequality follows from \ref{variousirr} and \ref{irredundantfunctions}. 
For the second consider a nice biorthogonal
 system $(f_\alpha, \delta_{x_\alpha}-\delta_{y_\alpha})_{\alpha<\kappa}$.
Consider open $V_\alpha, U_\alpha\subseteq K$ such that 
$|f_\alpha(x)-f_\alpha(x_\alpha)|<1/2$ for $x\in V_\alpha$ and
$|f_\alpha(x)-f_\alpha(y_\alpha)|<1/2$ for $x\in V_\alpha$. 
Note that $x_\beta\in V_\alpha$ and $y_\beta\in U_\alpha$ would give that
$|f_\alpha(x_\beta)-f_\alpha(y_\beta)|>0$ contradicting the biorthogonality.
So $V_\alpha\times U_\alpha$ witnesses that $(x_\alpha, y_\alpha)_{\alpha<\kappa}$
is a discrete subspace of $K^2$.
\end{proof}

The above result was used to conclude that any strong $S$-space has countable
irredundance, this applies to the Kunen space as well as to that of \cite{brechkoszmider}.
 However, the first construction of countably irredundant and uncountable
Boolean algebra (assuming $\diamondsuit$)was due to Rubin(\cite{rubin}) and seems not to have applications in Banach spaces.

Now we are left with the invariants $irr$, $nbiort_2$,  $s^2$, $biort$ and we ask if the inequalities
between them are strict. As before we have no absolute examples but a few consistent examples required
considerable work:

\begin{theorem}[Ros{\l}anowski, Shelah \cite{roslan}] It is 
consistent that there is a Boolean algebra such that
$$\omega=irr(A)<s(K_{\mathcal A}^2)=\omega_1.$$
\end{theorem}

\begin{theorem}[Brech, Koszmider \cite{brechkoszmider2}] It is 
consistent that there is a Boolean algebra $A$  such that
$$\omega=irr(A)=nbiort_2(K_{A})=s(K_{\mathcal A}^2)<biort(K_{\mathcal A})=\omega_1.$$
\end{theorem}

The above examples on the level of $\omega$ and $\omega_1$ cannot be
obtained in ZFC because we have

\begin{theorem}[Todorcevic \cite{stevoirr}, \cite{stevobio}]
Assume Martin's axiom and the negation of the CH. Suppose $\mathcal A$
is an uncountable Boolean algebra, then $irr(\mathcal A)$ is uncountable.
\end{theorem}
However we do not know the answer to the following:
\begin{question} Is any of the following equalities true, where
 $K$ stands for compact
Hausdorff space and $A$ for a  Boolean algebra:
\begin{enumerate}
\item $irr(A)=nbiort_2(K_{A})$?
\item $nbiort_2(K_{A})=biort_2(K_{A})$?
\item $nbiort_2(K_{A})=s(K_{A}^2)$?
\end{enumerate}
\end{question}
Natural questions related to Todorcevic's result are:
\begin{question} Assume Martin's axiom and the negation of CH. Suppose
that $K$ is a nonmetrizable compact space. Does $C(K)$ have an uncountable nice biorthogonal system?
\end{question}
\begin{question} Does Martin's axiom imply that
 $irr(\mathcal A)=|\mathcal A|$ for any Boolean algebra
of cardinality less than continuum?
\end{question}
The above statement is consistent as proved in \cite{stevoirr} Proposition 2.
\begin{theorem}[C. Brech, P. Koszmider \cite{brechkoszmider2}]
For each natural $n>1$ it is consistent that there is a compact Hausdorff space 
$K_{2n}$ such that in $C(K_{2n})$ there is no uncountable $(2n-1)$-supported biorthogonal sequence 
 but there are $2n$-supported biorthogonal systems, i.e.,
$$\omega=biort_{2n-1}(K_{2n})<biort_{2n}(K_{2n})=\omega_1.$$
\end{theorem}
Here the parity of the integers involved plays an important role, and we do not know the answer
to the following:
\begin{question}Is it consistent that there is an integer $n>1$ and a compact space $K$ 
such that
$$biort_{2n}(K)<biort_{2n+1}(K)?$$
\end{question}

Beyond the cardinals $\omega$ and $\omega_1$ most of the important questions 
are unsolved as even the following is a well known open problem:
\begin{question} Is there an absolute example of a Boolean algebra $\mathcal A$ such that
 $irr(\mathcal A)<|\mathcal A|$?
\end{question}

\section{Semibiorthogonality   and hereditary Lindel\"of degree}\par

In this section we go well beyond versions of biorthogonality which we considered 
in the previous sections and consider a well-ordered as well as positive version
of it:
\begin{definition}
If $\kappa$ is
an ordinal, a trasfinite
sequence $(x_i,x_i^*)_{i< \alpha}\subseteq X\times X^*$ is called a semi-biorthogonal
sequence if and only if $x_i^*(x_i)=1$ for all $i\in I$ and $x^*_i(x_j)=0$ if $j<i<\alpha$
and  $x^*_i(x_j)\geq 0$ if $i<j<\alpha$. 
$$sbiort(X)=\sup\{|\kappa|:\  \hbox{there is a semibiorthogonal system}\
(x_i,\phi_i)_{i\in \kappa}\subseteq X\times X^*\}.$$
\end{definition}
Semibiorthogonal sequences became important after Borwein and Vanderwerff proved 
in \cite{borwein} that the existence of a support set and so the so called
Rolewicz's  problem from \cite{rolewicz} is equivalent to the existence of an uncountable
semibiorthogonal sequence.  
Similar procedure which is applied to the definitions of biorthogonal system and $biort$ to
obtain the definitions of a semibiorthogonal sequence and $sbiort$ can be applied to
other versions of biorthogonal sequences and biort. In particular we will consider
$sbiort_n$ for $n\in \N$.
There is a relationship between hL and versions of semibiorthogonality, namely:

\begin{theorem}[Lazar \cite{lazar}]\label{hlsbiorex1} Suppose $K$ is a compact  Hausdorff
space, then $$sbiort_1(C(K))=hL(K),$$
 and so $hL(K)\leq sbiort(C(K))$.
\end{theorem}
\begin{proof} We use \ref{proleftright}.
Suppose $\{x_\alpha: \alpha<\kappa\}\subseteq K$ and $\{U_\alpha: \alpha<\kappa\}$
is a sequence of open subsets of $K$ such that $x_\alpha\in U_\alpha$ and
$x_\beta\not\in U_\alpha$ for $\beta>\alpha$. Let $f_\alpha: K\rightarrow [0,1]$ be
a continuous function such that $f_\alpha(x_\alpha)=1$ and $f_\alpha\restriction (K\setminus U_\alpha)=0$.
Then $(f_\alpha, \delta_{x_\alpha})_{\alpha<\kappa}$ forms a $1$-supported semibiorthogonal
sequence. 

Now if $(f_\alpha, a_\alpha\delta_{x_\alpha})_{\alpha<\kappa}$ is   semibiorthogonal, then
$a_\alpha\not=0$. Consider $U_\alpha=\{x\in K: f_\alpha(x)>a_\alpha/2\}$. We have that
$x_\alpha\in U_\alpha$ and $x_\beta \not\in U_\alpha$ if $\beta>\alpha$, hence
$\{x_\alpha:\alpha<\kappa\}$ is not Lindel\"of, so $hL(K)\geq sbiort_1(K)$.
\end{proof}

Based on this
result and the fact that $hL(K)=w(K)$ for $K$ scattered (see \cite{monk} 25.130), using the Kunen 
space, and the space obtained in \cite{brechkoszmider} we may conclude the following:

\begin{corollary}$ $
\par
\begin{enumerate}
\item CH implies that there is a Banach space $X$ of the form $C(K)$ such that
$\omega=biort(X)<sbiort(X)=\omega_1=2^\omega.$
\item It is consistent that there is a Banach space $X$ of the form
$C(K)$ such that 
$\omega=biort(X)<sbiort(X)=\omega_2=2^\omega.$
\end{enumerate}
\end{corollary}

However there is no hope that $hL(B_{X^*})$ will play a nontrivial role as we have the following:
\begin{proposition}\label{HLBanach} Suppose
 $X$ is a Banach space with $dens(X)=\kappa$, then
there is a sequence $(x_\alpha,x^*_\alpha)_{\alpha<\kappa}$ such that
$x^*_\alpha(x_\beta)=0$ for $\beta<\alpha$ and
$x^*(x_\alpha)=1$. In particular 
$$hL(B_{X^*})=dens(X)$$ for any Banach space $X$.
\end{proposition}
\begin{proof}
Given a closed subspace in a Banach space $X$ and a vector which does not belong to
it, one can construct a functional which has value zero on the subspace and 
value one on the vector. This allows us to construct a transfinite sequence 
as in the proposition. Then weak$^*$ open sets
$U_\alpha=\{x^*: x^*(x_\alpha)>1/2\}$ witness the fact that $x_\alpha^*$s form
a right-separated sequence, and so \ref{proleftright} can be used to conclude
the proposition.
\end{proof}

The above result can be considered as a Banach space version of Lemma \ref{sneider}.
Most important recent result on semibiorthogonal sequences is the following:

\begin{theorem} [Koszmider; Lopez-Abad, Todorcevic]
There is a nonmetrizable  compact $K$ without uncountable semibiorthogonal
sequences. In particular such that $sbiort(C(K))<dens(C(K))$.
\end{theorem}
It shows that Rolewicz's problem is undecidable.
The space of \cite{piotrrolewicz} is a version of the split interval, and the example
of \cite{stevogeneric} sec. 8 of the form $C(K)$ is from \cite{bellginsburgstevo}.
The question whether assuming only CH there is a nonseparable
Banach space $X$ satisfying $sbiort(X)=\omega$ was posed in \cite{mirnaistvan}.
There the authors construct under CH a compact space $K$ having the topological
and measure theoretic properties of the space from \cite{piotrrolewicz} where
there are only uncountable semibiorthogonal sequences  of a special kind.
Thus the question remains open. A related construction is that of \cite{shelah1}.

Another very interesting result showing that the behaviour of $sbiort$ is different
than that of $biort$ is the following: 
\begin{theorem}[Todorcevic \cite{stevobio} Thm. 9] If a Banach space 
of the form $C(K)$ has density bigger than $\omega_1$,
then there is an uncountable  semibiorthogonal sequence in $C(K)$.
\end{theorem}
It is asked in \cite{stevobio} (Problem 5) if the above result can be generalized 
to an arbitrary Banach space.
This also generates the following
\begin{question} Is it true that for every
Banach space $X$ (of the form $C(K)$) we have
$$dens(X)\leq sbiort(X)^+?$$
\end{question}
We also have some results on finitely supported versions of semibiorthogonal sequences:
\begin{theorem} [C. Brech, P. Koszmider]
Suppose $n\geq 4$ is an even  integer.
It is consistent that there are compact spaces $K$
such that
$$\omega=sbiort_{n-1}(X)<biort_{n}(X)=\omega_1.$$
\end{theorem}
The following result shows again that almost biorthogonal sequences
are quite far away from the rest. We can even have a Banach space 
without support sets but with uncountable almost biorthogonal systems:

\begin{proposition}[Lopez-Abad, Todorcevic, \cite{stevogeneric}] It is consistent that there
is a Banach space $X$ such that 
 $sbiort(X)=biort(X)=\omega<abiort(X)=\omega_1$.
In particular that $sbiort(X)<s(B_{X^*})$.
\end{proposition}

\bibliographystyle{amsplain}

\end{document}